\documentclass[a4paper,11pt]{article}
\usepackage[sc]{mathpazo}
\usepackage[T1]{fontenc}
\usepackage{amsthm,amsmath, amsfonts,hyperref,a4wide}
\usepackage[usenames,dvipsnames]{xcolor} 
\usepackage{tikz}
\makeatletter
\tikzset{my loop/.style =  {to path={
  \pgfextra{}
  [looseness=12,min distance=6mm]
  \tikz@to@curve@path},font=\sffamily\small
  }}  
\makeatletter 
\usetikzlibrary{matrix,arrows,backgrounds,positioning,fit,shapes}

\newtheorem{theorem}{Theorem}[section]
\newtheorem{lemma}[theorem]{Lemma}
\newtheorem{cor}[theorem]{Corollary}

\theoremstyle{definition}

\theoremstyle{remark}
\newtheorem{rem}{Remark}

\newtheorem*{ack}{Acknowledgements}

\newcommand*{\R}{\mathbb{R}}

\newcommand*{\N}{\mathbb{N}}

\newcommand*{\C}{\mathbb{C}}

\newcommand*{\A}{\mathcal{A}}

\newcommand*{\W}{\mathcal{W}}

\newcommand*{\aut}{\mathrm{Aut}}

\newcommand*{\eps}{\varepsilon}
\renewcommand*{\l}{\lambda}

\title{Regularity lemmas in a Banach space setting}
\author{Guus Regts\footnote{University of Amsterdam. Email: \texttt{guusregts@gmail.com}.}}
\begin{document}
\maketitle

\begin{abstract}
\noindent Szemer\'edi's regularity lemma is a fundamental tool in extremal graph theory, theoretical computer science and combinatorial number theory.
Lov\'asz and Szegedy \cite{LS7} gave a Hilbert space interpretation of the lemma and an interpretation in terms of compactness of the space of graph limits.
In this paper we prove several compactness results in a Banach space setting, generalising results of Lov\'asz and Szegedy \cite{LS7} as well as a result of Borgs, Chayes, Cohn and Zhao \cite{BCCZ14}.
\end{abstract}

\section{Introduction}
\subsection{The regularity lemma}
Szemer\'edi's regularity lemma \cite{S76} is a fundamental tool in extremal graph theory, theoretical computer science and combinatorial number theory. See \cite{KS96} for a survey. The lemma has many interpretations, variations and extensions. See for example \cite{K97,FK99,AFKS00,RS04,DKKV05,G05,C06,T06,LS7,S11,BCCZ14}. 

Very roughly the lemma says something of the form: for each $\eps>0$ there exists $k\in \N$ such that the vertex set of any graph can be partitioned into at most $k$ parts, such that for `almost' all pairs of parts the edges between that pair of parts  behaves `almost' like a random bipartite graph, where `almost' depends on $\eps$.
The weak regularity lemma of Frieze and Kannan \cite{FK99} weakens the requirements of the partition in the regularity lemma and measures the error of approximation with respect to the cut norm. 
This has as a consequence that the constant $k$ can be taken to be much smaller.
From the perspective of the adjacency matrix of a graph this means that one approximates this matrix with a bounded sum of cut matrices (in particular this gives a low rank approximation) such that their difference is small with respect to the cut norm.
This is exactly the point of view we take in this paper: we want to find various types of low rank approximations to matrices and tensors, when measured in a particular norm.

Our work is inspired by the work of Lov\'asz and Szegedy \cite{LS7} and Borgs, Chayes, Cohn and Zhao \cite{BCCZ14} relating the compactness of the space of graph limits to Szemer\'edi's regularity lemma. 
We refer to the book by Lov\'asz \cite{L12} for more details on graph limits.
In \cite{LS6} Lov\'asz and Szegedy used the weak version of the regularity lemma \cite{FK99} to assign a limit object to a convergent sequence of dense graphs. This limit object is no longer a graph, but a symmetric measurable function $W:[0,1]^2\to [0,1]$, called a \emph{graphon}.
In \cite{LS7} Lov\'asz and Szegedy showed that the space of graphons, equipped with the cut metric is compact, interpreting this result as an analytical form of the regularity lemma.
Their compactness result implies various kinds of regularity lemmas varying from weak to very strong.
It has recently been extended by Borgs, Chayes, Cohn and Zhao \cite{BCCZ14} to the space of $\R$-valued functions $W$ with bounded $p$-norm, the \emph{$L^p$-graphon space} (for any fixed $p>1$).
\subsection{Compactness}
We will now describe the compactness of the graphon space,  which is denoted by $\W$, more precisely (for details concerning definitions we refer to the next section), after which we mention some of the results in the present paper.

Let $W:[0,1]^2\to \R$ be an integrable function. Consider for $p,q\in [1,\infty]$, $W$ as a kernel operator $W:L^p([0,1])\to L^q([0,1])$; that is, for $f\in L^p([0,1]),$ $(Wf)(x):=\int_{[0,1]} W(x,y)f(y)d\lambda$. The $p\mapsto q$-operator norm is defined by 
\[
\|W\|_{p\mapsto q}=\sup_{\|f\|_p=1} \|Wf\|_q,
\]
where $\|\cdot\|_s$ denotes the $s$-norm on the space $L^s([0,1])$. 
The norm $\|\cdot\|_{\infty \mapsto 1}$ is equivalent to the cut norm.
Define an equivalence relation $\sim$ on $\mathcal{W}$ as follows: $W\sim W'$ if for each $\eps>0$ there exists a measure preserving bijection $\tau:[0,1]\to [0,1]$ such that
$\|W-\tau W'\|_{\infty \mapsto 1}\leq \eps$ for $W,W'\in \mathcal{W}$.
Then the result of Lov\'asz and Szegedy \cite{LS7} can be stated as follows: 
\begin{equation}
\text{the space } (\mathcal{W},\|\cdot\|_{\infty\mapsto 1})/\sim \text{ is compact}.	\label{eq:compact LS}
\end{equation}
The result of Borgs, Chayes, Cohn and Zhao \cite{BCCZ14} then says that we can replace $\cal{W}$ with the symmetric functions in $L^p([0,1]^2)$ of norm at most $1$ for any fixed $p>1$.

In this paper we will show that in \eqref{eq:compact LS} we can also replace the norm $\|\cdot\|_{\infty\mapsto 1}$ by the norm $\|\cdot \|_{q\mapsto\frac{q}{q-1}}$, if we replace $\W$ by the unit ball of $L^p([0,1]^2)$, provided that  $p>\frac{q}{q-1}$, cf. Theorem \ref{thm:compact orbit spaces L}.
In fact, we generalise \eqref{eq:compact LS}, replacing the space $\cal{W}$ by a special weakly compact subset of a Banach space $X$, the relation $\sim$ by an equivalence relation obtained from a subgroup of the group of autmorphisms of $X$ and the norm $\|\cdot \|_{\infty\mapsto 1}$ by an operator-type norm, cf. Theorem \ref{thm:compact banach}.
From this result it is then easy to derive the results of Borgs, Chayes, Cohn and Zhao.
In Section \ref{sec:L} we will also utilise it to include $q\mapsto \frac{q}{q-1}$-norms and apply it to higher order tensors. In Section \ref{sec:ell} we will apply it to $\ell^p$-spaces.

Our method is based on work of the author and Schrijver \cite{RS12}. 
In \cite{RS12} the compactness result of Lov\'asz and Szegedy was extended to a general Hilbert space setting, putting emphasis on the possibility of using different norms than the cut norm and the use of groups and moreover using a different method of proof.
Consequently, our proof of Theorem \ref{thm:compact banach} does not use the martingale convergence theorem.
Thus it yields a different proof of the compactness result of Borgs, Chayes, Cohn and Zhao \cite{BCCZ14}. 
However, there are some similarities.
To prove Theorem \ref{thm:compact banach} we need a result from \cite{RS12}, cf. Lemma \ref{lem:weak Szem}, which may be viewed as a weak regularity lemma in a Hilbert space setting, and which generalises weak regularity results from \cite{FK99,DKKV05}.

\subsection{Algorithms and applications}
Some of the existing versions of the weak and strong regularity lemmas come with efficient algorithms for finding a low rank approximation (or regularity partition).
These algorithms have been applied to find approximation schemes for various sorts of dense instances of counting and optimisation problems \cite{FK99,DKKV05} and to property and parameter testing \cite{AFKS00,BCLSSV06,BCLSV08}; see also the book by Lov\'asz \cite{L12}.

Some of our results can also be put to algorithmic use. In particular, in the $\ell^p$ setting, Lemma \ref{lem:approx} can be used to give approximation algorithms for certain instances of computing the matrix $p\mapsto q$ norm and for finding approximate Nash-equilibria in two player games, in a similar spirit as has been done by Barman in \cite{B14}.
In the $L^p$ setting sampling algorithms from Borgs, Chayes, Lov\'asz, S\'os and Vesztergombi \cite{BCLSV08} can be applied to find low rank approximations to matrices yielding polynomial time approximation algorithms for computing the matrix $p\mapsto q$ norm for dense matrices.
This is work in progress and we will report on it in a forthcoming paper \cite{R15}.

\subsection{Organisation}
In the next section we will discuss some preliminaries and set up some notation. 
In Section \ref{sec:compact banach} we will state and prove Theorem \ref{thm:compact banach}, the aforementioned generalisation of \eqref{eq:compact LS}. We will also deduce some consequences from it.
In Section \ref{sec:L} we will apply this theorem to $L^p$ spaces and in Section \ref{sec:ell} to $\ell^p$ spaces.

\section{Preliminaries and notation}\label{sec:prel}
In this section we will give some preliminaries on Lebesgue spaces and set up some notation.
We refer to \cite{C90} for functional analytic background and to \cite{H50} for measure theoretic background.
\\

\noindent{\bf Lebesgue spaces}
For a measure space $(\Omega,\mathcal{A},\mu)$ and $p\in [1,\infty]$ we denote by $L^p(\Omega)$ the linear space of equivalence classes of $\mu$-integrable complex (or real) valued functions $f:\Omega\to \C \text{ (or } \R)$ with bounded $p$-norm, which is defined as
\begin{align*}
&\|f\|_p:=\left(\int|f|^pd\mu\right)^{1/p} \text{ for }p<\infty\text{, and}
\\
&\|f\|_{\infty}:=\inf_{t\geq0}\{|f(x)|\leq t\mid \text{ for $\mu$-almost all $x$}\}.
\end{align*}
(Recall that two functions are equivalent if they are equal $\mu$-almost everywhere.)

For our results it often does not matter whether we use real or complex-valued functions.
So we generally do not distinguish between the complex and real-valued cases. 
If however we want to specify that we work over the field of real numbers, we denote this by $L^p_\R(\Omega)$.
We often omit the reference to the sigma-algebra $\mathcal{\A}$ and the measure $\mu$. 
In case $\Omega=[0,1]^l$ for some $l\in \N$ we will always equip it with the Borel (or Lebesgue) sigma algebra and with the Lebsegue measure $\lambda$.

For a set $\Omega$, and $p\in[1,\infty]$, $\ell^p(\Omega)$ is just equal to $L^p(\Omega)$ with $\mathcal{A}$ the power set of $\Omega$ and $\mu$ the counting measure.
We often write $L^p$ and $\ell^p$ whenever the underlying space is clear.

For a normed space $(Y,\|\cdot\|)$ we denote its closed unit ball by $B(Y,\|\cdot\|)$, which is defined as $\{y\in Y\mid \|y\|\leq 1\}$. Often we just write $B(Y)$.
Let $(\Omega,\mu)$ be a probability space, i.e., $\mu(\Omega)=1$. 
An important property of the space $L^p(\Omega)$, that we will often use, is the nesting of the closed unit balls: for any $1<p<q<\infty$ we have
\begin{equation}
B(L^\infty)\subset B(L^q)\subset B(L^p)\subset B(L^1).	\label{eq:nesting Lp}
\end{equation}
For the closed unit balls in $l^p(\Omega)$, for any set $\Omega$, the opposite inclusions hold:
\begin{equation}
B(\ell^1)\subset B(\ell^p)\subset B(\ell^q)\subset B(\ell^\infty).	\label{eq:nesting lp}
\end{equation}
\noindent{\bf Weak topologies}
We often need weak topologies in this paper. So it wil be convenient to introduce some notation.
Let $(X,\|\cdot\|)$ be a normed space. Let $X^*$ be the dual space of $X$, i.e., the space of all continuous linear functions $f:X\to \C$ (or $\R$).
Let $Y\subset X^*$ be a subset of $X^*$.
The \emph{weak topology induced by $Y$} on $X$ is generated by sets of the form $B_{x_0,f,\eps}=\{x\in X\mid |f(x)-f(x_0)|<\eps\}$ for $x_0\in X$, $f\in Y$ and $\eps>0$.
If we do not specify $Y$ we mean the weak topology induced by $X^*$.

For $p\in [1,\infty]$ let 
\[p^*:=\left \{\begin{array}{cl}\frac{p}{p-1} & \text{if } 1<p<\infty,
								\\ 1 &\text{if } p=\infty,
								\\ \infty &\text{if }p=1.
								\end{array}\right.
								\]
It is well known that $L^{p^*}$ isometrically embeds into $(L^p)^*$ and this embedding is onto if $p<\infty$ (which holds for $p=1$ provided that $\mu$ is sigma finite).
So if $\Omega$ is a probability space, we can consider the weak topology on $B(L^p)$ induced by $L^{q}$ for each $q\geq p^*$ (by \eqref{eq:nesting Lp}). We denote this by $(B(L^p),w_{q})$.
For $\ell^p$ we need to take $q\leq p^*$. This is then denoted by $(B(\ell^p),w_{q})$.
\\

\noindent {\bf Group actions}
If a group $G$ acts on a set $S$ this induces an action on the functions from $S$ to $\C$ (or $\R$)  via $gf(s):=f(g^{-1}s)$ for $g\in G$, a function $f$ and $s\in S$. Moreover, $G$ has a natural action on $S^l$ for any $l\in \N$.
In particular, if $(\Omega,\mathcal{A},\mu)$ is a measure space and a group $G$ acts on $\Omega$ such that $\mu(gA)=\mu(A)$ for all measurable sets $A\in \mathcal{A}$ and $g\in G$ (in which case we call elements of $G$  \emph{measure preserving bijections}), then $G$ acts on $X=L^p(\Omega^l)$ for any $p$ and $l$ and preserves the $p$-norm.
If a group acts on a set $S$ and $T\subset S$ we denote by $GT$ the set $\{gt\mid g\in G, t\in T\}$.
We call $T\subseteq S$ \emph{$G$-stable} if $GT=T$.

\section{Compact orbit spaces in Banach spaces}\label{sec:compact banach}
In this section we will state and prove our main results concerning compact orbit spaces in Banach spaces and discuss some consequences, which may be viewed as regularity lemmas in a Banach space setting.
\subsection{The main theorems}
Before we can state our results, we need some definitions.
Let $X=(X,\|\cdot \|)$ be a normed space and let $R$ be a bounded subset of $X^*$.
We define a seminorm $\|\cdot\|_R$ and pseudo metric $d_R$ on $X$ by
\begin{equation}
\|x\|_R:=\sup_{r\in R}|r(x)|\quad \quad d_R(x,y):=\|x-y\|_R	\label{eq:def Rnorm}
\end{equation}
for $x,y\in X$.

For a pseudo metric space $(X,d)$ let $\text{Aut}(X)$ denote the group of invertible maps $g:X\to X$ that preserve $d$. 
Let $G$ be a subgroup of $\aut(X)$. Define a pseudo metric $d/G$ on $X$ by
\[
(d/G)(x,y):=\inf_{g\in G}d(x,gy)
\]
for $x,y\in X$.
Note that since $(d/G)(x,y)$ is just equal to the distance between the $G$-orbits of $x$ and $y$, this implies that $d/G$ is indeed a pseudo metric.
For our purposes it is sometimes more convenient to work with $(X,d/G)$ than with $X/G$, but note that $(X,d/G)$ is compact if and only if $X/G$ is compact.
Recall that a (pseudo) metric space is called \emph{totally bounded} if for each $\eps>0$ it can be covered with finitely many balls of radius $\eps$.

We can now state our first result about compactness of orbit spaces in Banach spaces, which we will prove in Section \ref{sec:proofs}.
\begin{theorem}\label{thm:compact banach v1}
Let $(X,\|\cdot\|)$ be a Banach space and let $G$ be a subgroup of $\aut(X)$.
Let $R\subseteq B(X^*)$ be $G$-stable, and let $W\subset X$ be $G$-stable and weakly compact.
If $(W,d_R/G)$ is totally bounded, then $(W,d_R/G)$ is compact.
\end{theorem}

Showing that $(W,d_R/G)$ is totally bounded may not be very simple.
However, if $W$ is somehow `close' to a Hilbert space (as will be made precise below) and $R$ can be embedded into a  bounded subset of this Hilbert space, then totally boundedness of $(W,d_R/G)$ can be deduced from compactness of the space of sums of elements from $R$ modulo $G$.

For a subset $Y$ of a linear space $X$ and $k\in \N$ we define 
\[
k\cdot Y:=\{y_1+\ldots+y_k\mid y_i\in Y\}.
\]
Note that when $Y$ is convex, $k\cdot Y$ is just equal to $kY$.
Let $X$ be a normed space and let $H$ be a Hilbert space.
We call $W\subset X$ \emph{$H$-small}\footnote{The notion of $H$-smallness can be seen as a qualitative refinement of relative weak compactness, by a classical result of Grothendieck \cite{G52}. Moreover, related notions have been used elsewhere in functional analysis, cf. \cite{DFJP74}.}, if here exists a contractive linear map $T:H\to X$ and a function $c:(0,\infty)\to \N$ such that $W\subset c(\eps)T(B(H))+\eps B(X)$ for each $\eps>0$. 
Note that $T$ gives rise to a contractive linear map $T^*:X^*\to H^*$, \emph{ the adjoint of $T$}, defined by $f\mapsto (h\mapsto f(T(h)))$ for $f\in X^*$ and $h\in H$. 
Identifying $H$ with $H^*$ (which we will always do) this gives a contractive linear map $TT^*:X^*\to X$.
When we talk about a $H$-small space we will implicitly assume the presence of the maps $T,T^*$ and $c$.

\begin{theorem}\label{thm:compact banach}
Let $(X,\|\cdot\|)$ be a Banach space and let $G$ be a subgroup of $\aut(X)$.
Let $R\subseteq B(X^*)$ be $G$-stable and let $W\subset X$ be $G$-stable and weakly compact.
If there exists a Hilbert space $H$ such that $W$ is $H$-small and if $(k\cdot (TT^*(R)),d_R/G)$ is compact for each $k\in \N$, then $(W,d_R/G)$ is compact.
\end{theorem}
Observe that when $X$ is a Hilbert space, Theorem \ref{thm:compact banach} reduces to \cite[Theorem 2.1]{RS12}.

Before we give a proof of  Theorem \ref{thm:compact banach}, let us remark that the compactness of the $L^p$-graphon space proved by Borgs, Chayes, Cohn and Zhao \cite{BCCZ14} follows almost immediately from it.
Let $R\subset L^{\infty}([0,1]^2)$ be defined by 
\[
R:=\{\chi_{A\times B}\mid A,B\subseteq [0,1] \text{ measurable} \}.
\]
This makes $\|\cdot \|_R$ into the cut norm.
Let us denote the group of measure preserving bijections $\tau:[0,1]\to[0,1]$ by $S_{[0,1]}$.
Then $d_R/S_{[0,1]}$ is equal to $\delta_\square$, the cut metric. 
The result of Borgs, Chayes, Cohn and Zhao is then equivalent to the following:
\begin{cor}\label{cor:BCCZ}
Let $p>1$. Then $(B(L^p([0,1]^2)),d_R/S_{[0,1]})$ is compact.
\end{cor}
\begin{proof}
In case $p\geq 2$ we take $X=L^2$, otherwise we set $X=L^1$.
Write $W=B(L^p)\subset X$ and let $H=L^2$. Then $H$ is a Hilbert space and $W$ is $H$-small by \cite{BCCZ14}, or see Lemma \ref{lem:bound1L} (the map $T$ being the identity.) Note also that $T^*T$ restricted to $R$ is the identity.
Since any measurable set $A$ can be mapped onto any interval of length $\l(A)$ by a measure preserving bijection, cf. \cite{N98}, it follows that $(R^k,d_R/S_{[0,1]})$ is compact and hence $(k\cdot R,d_R/S_{[0,1]})$ as well, as it is a continuous image of $(R^k,d_R/S_{[0,1]})$ (see \cite{RS12}).
Let us first assume that $p<\infty$. Then, as $(B(L^p),w_{p^*})$ is compact by the Banach-Alaoglu theorem, it follows that $(B(L^p),w_\infty)$ is compact. (Since $w_\infty$ induces a weaker topology than $w_{p^*}$.) 
In case $p=\infty$, we may use that $B(L^\infty)\subset B(L^{s})$, for any fixed $s>1$, is a closed convex set with respect to the $s$-norm. 
This implies that $(B(L^\infty),w_{s^*})$ is compact and hence $(B(L^\infty),w_{\infty})$ is compact, as desired.
Theorem \ref{thm:compact banach} now implies that $(B(L^p),d_R/G)$ is compact.
\end{proof}

With little additional effort we can derive something similar as Corollary C.8 in \cite{BCCZ14}:
\begin{cor}
Given a function $\kappa:(0,\infty)\to (0,\infty)$, let 
\[W_\kappa:=\bigcap_{\eps>0} \kappa(\eps)B(L^{\infty}_\R([0,1]^2))+\eps B(L^1_\R([0,1]^2)).\]
Then $(W_\kappa,d_R/S_{[0,1]})$ is compact.
\end{cor}
\begin{proof}
Throughout the proof we assume that we work over the real numbers.
Let again $H=L^2$ and $X=L^1$. Then $W_\kappa\subset X$ is clearly $H$-small by \eqref{eq:nesting Lp}.
So by Theorem \ref{thm:compact banach} it suffices to check that $(W_\kappa,w_{\infty})$ is compact.
We may assume that $W_\kappa$ is contained in $B(L^1)$. 
Since $(B(L^\infty)^*,w_{\infty})$ is compact by the Banach-Alaoglu theorem and contains $B(L^1)$, it suffices to show that $W_\kappa$ is weakly closed.

To this end let $\mu$ be an element of the weak closure of $W_\kappa$.
Then $\mu$ is a finitely additive signed measure on $[0,1]^2$.
We will show that $\mu$ is in fact sigma additive.
To see this, let $\eps>0$ and choose $\delta=\eps\kappa(\eps)^{-1}$. 
Then for each measurable set $U\subset [0,1]^2$ there exists $w\in W_\kappa$ such that $|\mu(U)-\int_U wd\lambda|\leq \eps$.
This implies that if $\lambda(U)<\delta$, then $\mu(U)\leq 3\eps$, as for each $w\in W_\kappa$, $|\int_U wd\lambda| \leq\lambda(U)\kappa(\eps)+\eps=2\eps$.
Let now $(U_n)_{n\in \N}$ be a collection of pairwise disjoint measurable sets. 
Write $U=\cup_{n\in \N}U_n$. 
Then, by sigma additivity of the Lebesgue measure, we have $\lambda(\cup_{n>N}U_n)\to 0$, as $N$ tends to infinity.
This implies that $\mu(\cup_{n>N}U_n)\to 0$, as $N$ tends to infinity.
So $\mu(U)=\sum_{n\in \N}\mu(U_n)$, showing that $\mu$ is sigma additive.

Since $\mu$ is finite, the Radon-Nikodym theorem implies the existence of an integrable function $f$ such that $\int_{U}fd\lambda=\mu(U)$ for each measurable $U$.

We will now show that $f$ is again contained in $W_{\kappa}$.
Let $\eps>0$ and define $U:=\{x\mid f(x)>\kappa(\eps)\}$, $V:=\{x\mid f(x)<-\kappa(\eps)\}$ and let $\delta >0$.
Then there exists $w\in W_\kappa$ such that $|\int_U(f-w)d\l|\leq\delta$ and $|\int_V(f-w)d\l|\leq \delta$.
This implies
\begin{align}
\int_{U\cup V}|f|d\l&=\int_{U}fd\l-\int_V fd\l\leq \int_U wd\l-\int_V wd\l+2\delta	\nonumber
\\
&\leq \int_{U\cup V}|w|d\l +2\delta \leq \lambda(U\cup V)\kappa(\eps)+\eps+2\delta.\label{eq:W weak closed}
\end{align}
Since \eqref{eq:W weak closed} holds for any $\delta>0$, this shows that $f\in W\kappa$ and finishes the proof.
\end{proof}
\begin{rem}
The requirement that we work over the real numbers in the corollary above can be omitted if we choose to define $W_\kappa$ in the following way:
\[
W_\kappa:=\bigcap_{\eps>0}(\kappa)(\eps)\big(B(L^\infty_\R([0,1]^2))+i B(L^\infty_\R([0,1]^2))\big)+\eps\big(B(L^1_\R([0,1]^2))+iB(L^1_\R([0,1]^2))\big).
\]

\end{rem}
\subsection{Proofs of Theorem \ref{thm:compact banach v1} and \ref{thm:compact banach}}\label{sec:proofs}
The proofs of Theorems \ref{thm:compact banach v1} and \ref{thm:compact banach} follow the same pattern as the proof of Theorem 2.1 in \cite{RS12}.
Let us start with the proof of Theorem \ref{thm:compact banach v1}.

\begin{proof}[Proof of Theorem \ref{thm:compact banach v1}]
Since $(W,d_R/G)$ is totally bounded it suffices to show that $(W,d_R/G)$ is complete.
Let $a_1,a_2,\ldots$ be a Cauchy sequence in $(W,d_R/G)$. We may assume that $(d_{R}/G)(a_i,a_{i+1})< 2^{-i}$ for each $i$. 
Next find iteratively $g\in G$ such that $d_{R}(a_i,ga_{i+1})<2^{-i}$ and replace $a_{i+1}$ by $ga_{i+1}$.
Then $d_{R}(a_i,a_{j})<2^{-i+1}$ for all $i$ and $j\geq i$ and hence this sequence is Cauchy with respect to $d_{R}$.
We will show that this sequence is also convergent. 

Let for $n\in \N$, $A_n$ denote the weak closure of the set $\{a_n,a_{n+1},\ldots\}$.
By weak compactness, $A:=\cap_{n\in \N} A_n$ is not empty.
Let $a\in A$.
We will show that $a_n$ converges to $a$ with respect to $d_{R}$.
Let $\eps>0$.
Choose $N$ such that for all $n,m>N$, $d_{R}(a_n,a_m)\leq \eps$.
We will show that $d_R(a_n,a)\leq2\eps$ for all $n>N$.
To this end fix $r\in R$. 
Then there exists $m\geq N$ such that $|(r,a_m-a)|\leq \eps$.
Then for any $n>N$ we have
\[
|(r,a-a_n)|\leq |r,a-a_m)|+|(r,a_m-a_n)|\leq 2\eps,
\]
showing that $d_R(a_n,a)\leq 2\eps$. Hence $(a_n)$ converges to $a$ with respect to $d_R$. 
\end{proof}

To prove Theorem \ref{thm:compact banach}, we need the following lemma from \cite{RS12}, which may be viewed as a weak regularity lemma in a Hilbert space setting. In particular, choosing appropriate $R$, we recover results from \cite{FK99} and \cite{DKKV05} respectively.

\begin{lemma}[\cite{RS12}]\label{lem:weak Szem}
Let $H$ be any Hilbert space.
Let $R\subseteq B(H)$ be closed under multiplication by elements of $\{z\in \C\mid |z|\leq 1\}$ (by elements of $[-1,1]$ if $H$ is real). Then for any $k\in \N$,
\[
B(H)\subseteq k\cdot R +\frac{1}{\sqrt{k}}B(H,\|\cdot\|_R).
\]
\end{lemma}
For convenience of the reader we will give a proof.
\begin{proof}
Let us denote the inner product on $H$ by $\langle\cdot,\cdot\rangle$.
Let $a\in B(H)$. 
Write $a_0=a$.
If $\|a_i\|_R>\frac{1}{\sqrt{k}}$ for some $i\geq 0$, then there exists $r\in R$ such that $|\langle a_i,r\rangle|>\frac{1}{\sqrt{k}}$.
Set $a_{i+1}=a_i-\langle a_i,r\rangle r$.
Note that by induction we have that $a_{i+1}-a_0\in i\cdot R\subset k\cdot R$, as $|\langle a_i,r\rangle|\leq 1$ by Cauchy-Schwarz.
Then 
\[
\|a_{i+1}\|_2^2=\|a_i\|_2^2-2|\langle a_i,r\rangle|^2+ |\langle a_i,r\rangle|^2\|r\|^2_2=\|a_i\|^2_2-|\langle a_i,r\rangle|^2(2-\|r\|_2^2)< \|a_i\|_2^2-1/k,
\]
by definition of $r$ and the fact that $R\subseteq B(H)$.
By induction we have that $\|a_{i+1}\|_2^2<1-i/k$.
This shows that for some $i\leq k$ we must have $\|a_i\|_R\leq \frac{1}{\sqrt{k}}$ and hence finishes the proof.
\end{proof}

We can now prove Theorem \ref{thm:compact banach}. 
\begin{proof}[Proof of Theorem \ref{thm:compact banach}]
By Theorem \ref{thm:compact banach v1} it suffices to show that $(W,d_R/G)$ is totally bounded.
We may assume that $R$ is closed under multiplication by elements of $\{z\in \C\mid |z|\leq 1\}$ (elements of $[-1,1]$ if our spaces are real), as this does not changes the norm induced by $R$. Neither does it change the compactness of $(k\cdot (TT^*(R)),d_R/G)$ for any $k$.
To show that $(W,d_R/G)$ is totally bounded, choose $\eps>0$. 
By assumption there exists a Hilbert space $H$ such that $W$ is $H$-small.
This means there exists a contractive linear map $T:H\to X$ and a constant $c>0$ such that $W\subset cT(B(H))+\eps B(X)$. 
As $(k\cdot (TT^*(R)),d_R/G)$ is compact there exists a finite set $F\subset k\cdot (TT^*(R))$ such that $k\cdot (TT^*(R))\subset GF+\frac{\eps}{c} B(X,\|\cdot\|_R)$.
Setting $k:=\lceil c^2\eps^{-2}\rceil$, we have by Lemma \ref{lem:weak Szem}, since $T(B(H,\|\cdot\|_{T^*(R)}))\subseteq B(X,\|\cdot\|_R)$,
\begin{align*}
W&\subseteq cT(B(H))+\eps B(X)\subset c(k\cdot (TT^*(R))+k^{-1/2}T(B(H),\|\cdot\|_{T^*(R)}))+\eps B(X)\subseteq 
\\
&c(GF+\frac{\eps}{c}B(X,\|\cdot\|_R)+k^{-1/2} B(X,\|\cdot\|_R))+\eps B(X)\subseteq GcF+3\eps B(X,\|\cdot\|_R).
\end{align*}
This shows that $(W,d_R/G)$ is totally bounded and finishes the proof.
\end{proof}
\begin{rem}
The proof shows that if $R$ is already closed under multiplication by elements of $\{z\in \C\mid |z|\leq 1\}$, then it is enough to demand that $(k\cdot TT^*(R),d_R/G)$ is merely totally bounded.
\end{rem}

\subsection{Weak and strong regularity}
The following is implicit in the proof of Theorem \ref{thm:compact banach}, and may be viewed as an extension of the weak regularity lemma from the Hilbert space setting to a Banach space setting.

\begin{lemma}\label{lem:weak banach}
Let $(X,\|\cdot\|)$ be a Banach space, let $W\subset X$ and let $R\subseteq B(X^*)$.
Suppose there exist a Hilbert space $H$ such that $W$ is $H$-small.
Then for each $\eps>0$ there exists $k\leq \lceil c(\eps)/\eps \rceil^2$ such that for each $w\in W$ there exists $\alpha\in \{z\in \C\mid |z|\leq 1\}^l$ and $x_1,\ldots,x_l\in TT^*(R)$, with $l\leq k$, 
such that $\|w-c\sum_{i=1}^l \alpha_ix_i\|_R\leq 2\eps$.
\end{lemma}
In \cite{LS7}, Lov\'asz and Szegedy applied the compactness of the graphon space, cf. \eqref{eq:compact LS}, to derive an approximation result for graphons, cf. \cite[Lemma 5.2]{LS7}. This result implies several types of regularity lemmas varying from  the weak regularity lemma \cite{FK99}, to the original lemma \cite{S76}, to a `super strong' variant \cite{AFKS00}. See \cite{L12} for more details.
We can derive something similar in our Banach space setting:
\begin{lemma}\label{lem:cor compact}
Let $(X,\|\cdot\|)$ be a Banach space, let $G\subseteq \aut(X)$, let $R\subseteq B(X^*)$ and let $W\subset X$ be $G$-stable and suppose that $(W,d_R/G)$ is compact.
Let for $k\in \N$, $Y_k\subset W$ be $G$-stable such that $Y:=\cup_{k\in \N} Y_k$ is dense in $W$ (w.r.t. $\|\cdot \|$). 
Let $h:(0,\infty)\times \N\to (0,1)$ be any function. 
Then for any $\eps>0$ there exists $N\in \N$ such that for any $w\in W$ there exists $w'\in W$ 
and $y\in Y_{m}$, with $m\leq N$, such that
\[
\|w-w'\|_R \leq h(\eps,m) \quad \text{ and }\quad  \|w'-y\| \leq \eps.
\]
\end{lemma}

The proof goes along the same lines as the proof of Lemma 5.2 in \cite{LS7}.
\begin{proof}
We may assume that $h$ is monotonically decreasing in its second variable. 
Let $\eps>0$.
Since $Y$ is dense in $W$, for each $w\in W$, there exists $n\in \N$ and $y\in Y_n$ such that $\|w-y\|\leq \eps$. 
Let $f(w)$ denote the smallest $n$ such that there exists $y\in Y_n$ with $\|w-y\|\leq \eps$.
For $w\in W$ let $O(w)$ denote the $d_R/G$-open set defined as 
\[
O(w):=\{w'\in W\mid (d_{R}/G)(w,w')< h(\eps,f(w))\}.
\]
By compactness of $(W,d_R/G)$, it follows that there exists $w_1,\ldots,w_t$ such that the union of the $O(w_i)$ contains $W$.
This means that for any $w\in W$ there exists $i\in \{1,\ldots,t\}$ and $y\in Y_m$, with $m= f(w_i)$, such that $\|w_i-y\|\leq \eps$ and $(d_R/G)(w,w_i)< h(\eps,f(w_i))$.
Choosing $N:=\max_{i=1,\ldots,t} f(w_i)$ and a suitable $g\in G$ applied to both $w_i$ and $y$ we arrive at the desired conclusion.
\end{proof}

Since the norm $\|\cdot\|$ on $X$ satisfies $\|x\|\leq \|x\|_R$ for each $x\in X$, taking $h(\eps,m)=\eps$ for all $m$, Lemma \ref{lem:cor compact} implies several types of weak regularity lemmas by taking different choices of
$Y_k$. 
In particular, it allows to use different types of approximations not based on sums of elements from $\{z\in \C\mid |z|\leq 1\}TT^*(R)$ as in Lemma \ref{lem:weak banach}.
We will elaborate a bit more on this in the next two sections, where we will apply the results obtained here to the spaces $\ell^p(\N^l)$ and $L^p([0,1]^l)$.

\section{Compact orbit spaces in $L^s(([0,1]^l)$}\label{sec:L}
In this section we will apply Theorems \ref{thm:compact banach v1} and \ref{thm:compact banach} to the Banach space $X=L^s([0,1]^l)$ for some fixed $s\in [1,\infty)$ and $l\in \N$.
Let for any $q\in [1,\infty]$, 
\begin{align}
&R^q_l:=\{r_1\otimes \ldots\otimes r_l \mid r_i\in B(L^q([0,1]))\}\subset B(L^q([0,1]^l)).	\nonumber
\end{align}
Recall that $S_{[0,1]}$ denotes the group of measure preserving bijections $\tau:[0,1]\to [0,1]$.
We will derive the following result from Theorems \ref{thm:compact banach v1} and \ref{thm:compact banach}, generalising Corollary \ref{cor:BCCZ}, the compactness result of Borgs, Chayes, Cohn and Zhao \cite{BCCZ14}.

\begin{theorem}\label{thm:compact orbit spaces L}
Let $p\in (1,\infty]$ and let $q$ be such that $q>p^*$.
Fix $l \in \N$. Then the space $(B(L^{p}([0,1]^{l})),d_{R^q_l})/S_{[0,1]}$ is compact. 
\end{theorem}

We will prove Theorem \ref{thm:compact orbit spaces L} in Section \ref{sec:proof L} below. Let us first make some remarks and state some consequences.

Borg, Chayes, Cohn and Zhao \cite{BCCZ14}, who proved Theorem \ref{thm:compact orbit spaces L} for $l=2$ and $q=\infty$, already noted that the same result cannot be true when $p=1$. We note here that in case $p=2$ we really need that $q>2$ since $B(L^2([0,1]^2),d_{R^2_2})/S_{[0,1]}$ is not totally bounded. 
To see this take 
$f_i$ to be the function which is constant $i$ on the rectangle $[0,1/i]\times [0,1/i]$.
Since for rank $2$ functions $f$ we have $\|f\|_{R^2_2}\leq \|f\|_2\leq 2\|f\|_{R^2_2}$, we may as well replace $d_{R^2_2}$ by $d_2$.
Then for any $\pi\in S_{[0,1]}$ we have for $i=2^k<j=2^l$ with $l\geq k+2$,
\[
\int_{[0,1]^2}|f_i-\pi f_j|^2\geq \int_{[0,1]^2}|f_i-f_j|^2\geq \frac{(2^k-2^l)^2}{2^l}\geq 1-2^{k+1-l}\geq 1/2,
\]
implying that we need to take $q>2$ when $p=2$. This example of course generalises to any $l>2$ and $p,q$ such that $q=p^*$.

Janson \cite{J13} showed, using the Riesz-Thorin theorem, that in case $p=\infty$ (and $l=2$), Theorem \ref{thm:compact orbit spaces L} remains true if one replaces the $\|\cdot\|_{R^q_2}$ norm by the $p_0\mapsto q_0$ operator norm for any $p_0,q_0>1$. See \cite[Lemma E.6]{J13}.
We can also use the Riesz-Thorin theorem in our setting for $l=2$. (We need to work over the complex numbers though. This does not cause any problems for real-valued functions, as the complex and real operator norms are within a constant factor of each other.)
Let us take $p>1$ and $q>p^*$. 
Taking $f\in B(L^p([0,1]^2))$, we then have $\|f\|_{p^*\mapsto p}\leq 1$.
Define for $\theta \in (0,1)$ $p_{\theta}$ and $q_{\theta}$ by
\begin{equation}
\frac{1}{p_{\theta}}:=\frac{1-\theta}{p^*}+\frac{\theta}{q}\quad \text{ and } \quad \frac{1}{q_{\theta}}:=\frac{1-\theta}{p}+\frac{\theta}{q^*}.	\label{eq:define theta}
\end{equation}
Then by the Riesz-Thorin theorem (cf. \cite[Theorem 1.1.1]{BL76}) $\|f\|_{p_\theta\mapsto q_\theta}\leq \|f\|_{q\mapsto q^*}^\theta=\|f\|_{R^q_l}^\theta$. 
By Theorem \ref{thm:compact orbit spaces L} this implies the following:
\begin{cor}\label{cor:Riesz L}
Let $p,q\in (1,\infty]$ such that $q>p^*$. Let $\theta\in (0,1)$ and let $p_\theta$ and $q_\theta$ be defined as in \eqref{eq:define theta}.
Then $(B(L^p([0,1]^2)),\|\cdot\|_{p_\theta\mapsto q_\theta})/S_{[0,1]}$ is compact.
\end{cor}

By Lemma \ref{lem:cor compact} the following is a direct corollary to Theorem \ref{thm:compact orbit spaces L} and the previous corollary.

\begin{cor}\label{cor:cor compact L}
Let $p,q\in (1,\infty]$ such that $q>p^*$. Fix $l\in \N$.
Let for $k\in \N$, $F_k \subset B(L^p([0,1]^l))$ be $S_{[0,1]}$-stable such that $F:=\cup_{k\in \N} F_k$ is a dense subset of $B(L^p([0,1]^l))$ (w.r.t. $(\|\cdot\|_{q^*}$).
Let $h:(0,\infty)\times \N\to (0,1)$ be any function. 
Then for any $\eps>0$ there exists $N\in \N$ such that for any $g\in B(L^p([0,1]^l))$ there exists $g'\in B(L^p([0,1]^l))$ 
and $f\in F_{m}$, with $m\leq N$, such that
\begin{equation}
\|g-g'\|_{R^q_l}\leq h(\eps,m)\quad \text{ and }\quad \|g'-f\|_{q^*}\leq \eps .	\label{eq:strong}
\end{equation}
In case $l=2$ and $p_\theta$ and $q_\theta$ are defined by \eqref{eq:define theta}, the $R^q_l$ norm may be replaced by $\|\cdot\|_{p_\theta\mapsto q_\theta}$ in \eqref{eq:strong}.
\end{cor}

There are various choices for the sets $F_k$ in the corollary above. For example one can take $F_k$ to be those function that are sums of at most $k$ rectangles (a \emph{rectangle} is a function of the form $c\chi_{A_1\times \cdots \times A_l}$ with $A_i$ measurable and $c$ a constant.) 
A special case is to take $F_k$ to be those sums of rectangles that are coming from a partition of $[0,1]$ into at most $k$ sets.
 Another choice would be to take $F_k$ to be those functions that take only $k$ different values (i.e. step functions with $k$ steps).

\subsection{Proof of Theorem \ref{thm:compact orbit spaces L}} \label{sec:proof L}
To prove Theorem \ref{thm:compact orbit spaces L}, we need some preliminary results.

The next lemma implies that $B(L^p([0,1]^l))\subset L^1([0,1]^l)$ is $H$-small for $H=L^2([0,1]^l)$ and any $p\in(1,2]$. 

\begin{lemma}\label{lem:bound1L}
Let $(\Omega,\mu)$ be any probability space.
Let $p>p'\geq1$ and $\varepsilon >0$.
Then there exists a constant $C$ such that $B(L^p(\Omega))\subseteq CB(L^{\infty}(\Omega))+ \varepsilon B(L^{p'}(\Omega))$.
\end{lemma}
\begin{proof}
Define for convenience $B^s:=B(L^s(\Omega))$, for any $s\in [1,\infty]$.
Fix any $K\geq1$ and let $f\in B^p$. Let $A:=\{\omega\in\Omega\mid |f(\omega)|>K\}$.
Then $f=f_1+f_2$ with $f_1=f\chi_{\Omega\setminus A}$ and $f_2=f\chi_{A}$. 
Clearly, $f_1\in KB^{\infty}$.
Next we consider $\|f_2\|_{p'}$:
\[
\|f_2\|^{p'}_{p'}=\int_\Omega |f_2|^{p'}d\mu \leq \int_\Omega |f_2|^{p'} (|f|/K)^{p-p'}d\mu\leq K^{p'-p}\int_\Omega |f|^pd\mu\leq K^{p'-p}.
\]
Hence choosing $C$ in such a way that $C^{1-p/p'}=\varepsilon$ we obtain the desired result.
\end{proof}

\begin{lemma}\label{lem:bound 2}
Let $p>s\geq1$ and let $\varepsilon>0$. 
Then for any $k\in \N$ there exists a constant $C$ such that $R^p_k\subseteq CR^\infty_k+\varepsilon R^s_k$.
\end{lemma}
\begin{proof}
The proof is by induction on $k$.
The case $k=1$ follows by Lemma \ref{lem:bound1L}.
Now suppose $k>1$. Choose any $p'\in (s,p)$ and let $\delta>0$.
By induction, there exists a constant $C_1$ such that 
$R^p_{k-1}\subseteq C_1R^{\infty}_{k-1}+\delta R^{p'}_{k-1}$ and $R^p_1\subseteq C_1 R^\infty_1+\delta R^{p'}_1$.
Let us define for $s,t\in [1,\infty]$, 
\[
R^{s,t}_{k-1,1}:=\{f_1\otimes f_2\mid f_1\in R^{s}_{k-1}, f_2\in R^{t}_1\}.
\]
Then
\begin{equation}\label{eq:inclusion1}
R^p_k\subseteq C_1^2 (R^\infty_k+C_1\delta(R^{\infty,p'}_{k-1,1}+R^{p',\infty}_{k-1,1})+ \delta^2  R^{p'}_k.
\end{equation}
Next choose $\eta>0$. Then by induction there is a constant $C_2$ such that $R^{p'}_1\subseteq C_2R^\infty_1+\eta R^s_1$ and 
$R^{p'}_{k-1}\subseteq C_2R^{\infty}_{k-1}+\eta R^s_{k-1}$.
Plugging this into \eqref{eq:inclusion1} we obtain that
\[
R^p_k\subseteq CR^\infty_k+C_1\eta \delta (R^{\infty,s}_{k-1,1}+R^{s,\infty}_{k-1,1})+\delta^2R^{p'}_k,
\]
where $C:=C_1^2+2\delta C_1C_2$.
Finally, using that $R^t_l\subseteq R^s_l$ for any $t\geq s$ and $l\in \N$, and first choosing $\delta$ such that $\delta^2\leq 1/3\varepsilon$
and then $\eta$ such that $C_1\delta\eta\leq 1/3 \varepsilon$, we obtain the desired inclusion.
\end{proof}

\begin{proof}[Proof of Theorem \ref{thm:compact orbit spaces L}]
Let us write $R^s$ instead of $R^s_l$ for any $s\in [1,\infty]$.
Let us denote $S_{[0,1]}$ by $G$.
Define $X=L^{q^*}([0,1]^l)$.
Then $W=B(L^p([0,1]^l))\subset X$, as $p>q^*$.
The proof now proceeds in a number of steps. 

Let 
\[
R:=\{\chi_{A_1}\times \cdots \times \chi_{A_l}\mid A_i\subset [0,1] \text{ measurable for each }i\}.
\]
Note that $\|\cdot\|_{R}$ and $\|\cdot \|_{R^\infty}$ define equivalent norms on $L^1([0,1]^l)$. 
More precisely, for $f\in L^1([0,1]^l)$, we have 
\[
\|f\|_{R}\leq \|f\|_{R_\infty}\leq 4^l\|f\|_{R},
\]
since for the norm $\|\cdot\|_{R^\infty}$ we may restrict to functions $r=r_1\otimes \cdots \otimes r_l$ in \eqref{eq:def Rnorm} such that the $r_i$ only take values $-i,i,-1$ and $1$.
This implies that $(B(L^\infty),d_{R^\infty}/G)$ is compact by Corollary \ref{cor:BCCZ}, or more precisely by the proof of Corollary \ref{cor:BCCZ}, as the proof did not depend on $l$ being equal to $2$.

We will next show:
\begin{equation}
 (B(L^\infty),d_{R^q}/G) \text{ is compact. }\label{eq:compact step1}
\end{equation} 
To do so, let $(f_n)$ be a sequence in $B(L^\infty)$. 
By compactness of $(B(L^\infty),d_{R^\infty}/G)$ it has a convergent subsequence (which we may assume to be $(f_n)$ itself) that converges to some $f\in B(L^\infty)$ with respect to $d_{R^\infty}/G$.
So it suffices to show that $(f_n)$ converges to $f$ with respect to $d_{R^q}/G$.
Let $\eps>0$. Let $C$ be the constant from Lemma \ref{lem:bound 2} such that $R^q\subseteq \eps R^1+CR^\infty$.
Fix $N\in \N$ such that $(d_{R^\infty}/G)(f_n-f)<\eps/C$ for all $n\geq N$. In other words, for each $n\geq N$ there exists $g_n\in G$ such that $\|g_nf_n-f\|_{R^\infty}<\eps/C$.
Then, writing $f'_n=g_nf_n$,
\begin{align*}
\|f'_n-f\|_{R^q}=\sup_{r\in R^q}r(f'_n-f)\leq \sup_{\substack{r_1\in CR^\infty\\r_2\in \eps R^1}}r_1(f'_n-f)+r_2(f'_n-f)
\\
\leq C\|f'_n-f\|_{R^\infty}+\eps\|f'_n-f\|_{R^1}\leq \eps+2\eps=3\eps,
\end{align*}
showing that $f_n$ converges to $f$ with respect to $d_{R^q}/G$. This proves \eqref{eq:compact step1}.
We will now show: 
\begin{equation}
(B(L^p),d_{R^q}/G) \text{ is totally bounded in }X.\label{eq:tot bounded}
\end{equation}
To this end let $\eps>0$ and let $K$ be the constant from Lemma \ref{lem:bound1L} such that $B(L^p)\subseteq KB(L^\infty)+\eps B(L^{q^*})$ and note that $B(L^{q^*})=B(X)\subseteq B(X,\|\cdot\|_{R^q})$.
Write $\eps'=\eps/K$.
By compactness of $(B(L^\infty),d_{R^q}/G)$, there exists a finite set $F\subset B(L^\infty)$ such that 
$B(L^\infty)\subseteq GF+\eps'B(L^\infty,\|\cdot\|_{R^q})$.
Then 
\[
B(L^p)\subseteq KGF+K\eps'B(L^\infty,\|\cdot\|_{R^q})+\eps B(X,\|\cdot\|_{R^q})
\subseteq  KGF+2\eps B(X,\|\cdot\|_{R^q}),
\]
proving \eqref{eq:tot bounded}.

By the Banach-Alaoglu theorem we know that $(B(L^{p}),w_{p^*})$ is compact (here we assume that $p<\infty$, as this case is already covered).
This implies that $(B(L^p),w_{q})$ is compact, since $w_q$ yields a weaker topology than $w_{p^*}$, as $q>p^*$.
This finishes the proof by Theorem \ref{thm:compact banach v1}. 
\end{proof}

\section{Compact orbit spaces in $\ell^q(\N^l))$}\label{sec:ell}
In this section we will apply Theorem \ref{thm:compact banach v1} to the Banach space $X=\ell^q(\N^l)$ for $q\in (1,\infty]$.
Let $S_\N$ be the group of invertible maps $\tau:\N\to \N$.
We denote the metric induced by $\|\cdot \|_p$ by $d_p$ for any $p\in [1,\infty]$.
We will derive the following result:
\begin{theorem}\label{thm:compact orbit spaces}
Let $p\in [1,\infty)$ and let $q>p$. 
Fix $l\in \N$. Then the space $B(\ell^{p}(\N^{l})),d_q)/S_\N$ is compact.
\end{theorem}
We will prove this result in Section \ref{sec:proof ell} below. Let us first note that we really need $p<q$. 
Lemma \ref{lem:approx} below is not true for $p=q$ and is a direct corollary of Theorem \ref{thm:compact orbit spaces}. It seems that one can only take $q=p$ when both of them are equal to $2$, but then one has to use a different (weaker) metric and the orthogonal group instead of $S_{\N}$, cf. \cite{RS12}.

By Lemma \ref{lem:cor compact} the following is a direct corollary to Theorem \ref{thm:compact orbit spaces}:
\begin{cor}\label{cor:cor compact l}
Let $p\in [1,\infty)$ and let $q>p$. 
Fix $l\in \N$. 
Let for $k\in \N$, $X_k \subset B(\ell^p(\N^l))$ be $S_{\N}$-stable such that $X:=\cup_{k\in \N} X_k$ is a dense subset of $B(\ell^p(\N^l))$ (w.r.t. $\|\cdot\|_q$).
Then for any $\eps>0$ there exists $N\in \N$ such that for any $y\in B(\ell^p(\N^l))$ there exists $x\in X_{m}$, with $m\leq N$, such that $\|y-x\|_{q}\leq \eps$.

\end{cor}

We can take $X_k:=(k\cdot R)\cap B(\ell^p(\N^l))$ with $R=\{r_1\otimes \cdots \otimes r_l\mid r_i\in B(\ell^p(\N))\}$ in the corollary above. 
Then the corollary says that for each $\eps>0$ there is $N$ such that each $y\in B(\ell^q(\N^l))$
can be approximated in the $q$-norm by a tensor of rank at most $N$.
Unfortunately, Corollary \ref{cor:cor compact l} does not give any explicit bounds on $N$. 
It is not to be expected that the bounds on $N$ will be much better than the bound given by Lemma \ref{lem:approx} below (at least for $q=\infty$).
Indeed, for $q=\infty$, Alon, Lee, Schraibman and Vempala \cite{ALSV13} showed that for any $n\times n$ Hadamard matrix $M$ and $\eps\in (0,1)$ one has that for any matrix $M'$, such that $\|M-M'\|_{\infty} \leq \eps$, the rank of $M'$ is at least $(1-\eps^2)n$.
Since all entries of a Hadamard matrix are $1$ or $-1$, we need to take $p=\Omega(\log n)$, to make sure that $\|M\|_p$ is bounded. 
The bound given by Lemma \ref{lem:approx} on $N$ is then of order $\eps^{-\log n}$, which is polynomial in $n$.

\subsection{Proof of Theorem \ref{thm:compact orbit spaces}}\label{sec:proof ell}

For any $q\in [1,\infty]$, a subset $K\subset S$ and $x\in \ell^q(S)$ we define $x_K\in \ell^q(S)$ by  
\[x_K:=\left \{\begin{array}{c}x(i)  \text{ if } i\in K,\\
											0  \text{ otherwise}.\end{array}\right. \]
We need the following lemma for the proof of Theorem \ref{thm:compact orbit spaces}.

\begin{lemma}\label{lem:approx}
Let $S$ be any set.
Let $p\in [1,\infty)$ and let $q>p$. For any $\eps>0$ there exists $k=k(p,q,\eps)\in\N$ such that
for any $x\in B(\ell^p(S))$ there exists a set $K\subset S$ of size at most $k$ such that $\|x_K-x\|_{q}\leq \eps$.
In case $q=\infty$ we can take $k=\eps^{-p}$.
\end{lemma}
\begin{proof}
Let $C\subseteq S$ be a countable set of points, which we identify with $\N$, such that $x(i)=0$ for $i\notin C$.
We may assume that $x$ satisfies  $|x({1})|\geq |x({2})|\geq \ldots$.
Since $\|x\|_p\leq 1$, it follows that $|x(i)|^p\leq 1/i$ for all $i$. Let $K=\{1,\ldots,k\}$ for some $k\in \N$ to be fixed later.
Then for $q<\infty$,
\begin{equation}
\|x_K-x\|^{q}_{q}=\sum_{i\notin K} |x(i)|^{q}\leq \sum_{i>k}(1/i)^{q/p}.\label{eq:sum}
\end{equation}
It is a well known fact that the series $\sum_{i=1}^\infty (1/i)^s$ converges for every $s>1$.
This implies that for $k$ large enough, \eqref{eq:sum} will be bounded by $\eps^{q}$, which finishes the first part of the proof.

In case $q=\infty$, we have
\[
\|x_K-x\|_\infty=\sup_{i>k}|x(i)|\leq (1/i)^{1/p}.
\]
So taking $k$ such that $(1/k)^{1/p}\leq \eps$, we are done. 
\end{proof}

Note that Lemma \ref{lem:approx} is not true if one replaces $q$ by $p$.
One can simply take $x(i)=n^{-1}$ for $i=1,\ldots, n^p$ and $x(i)=0$ for $i>n^p$.
Then for any constant $k$ not depending on $n$ and any set $K\subset \N$ of size $k$ we have $\|x_K-x\|^p_p\leq(n^p-k)n^{-p}=1-kn^{-p}=1-o(1)$.

Despite the fact that Lemma \ref{lem:approx} is really simple, it can be utilised for algorithmic purposes, since we can go over all bounded size subsets of $\{1,\ldots,n\}$ in time polynomial in $n$.
This is used in \cite{R15} to give approximation algorithms for computing matrix $p\mapsto q$ norms.

Now we can give a proof of Theorem \ref{thm:compact orbit spaces}.
\begin{proof}[Proof of Theorem \ref{thm:compact orbit spaces}]
Write $W=B(\ell^p(\N^{l}))$. 
We may assume that $q<\infty$ since $d_\infty/S_\N$ yields a coarser topology than $d_{q'}/S_\N$ for any $q'<\infty$.
Let now $R=B(\ell^{q^*}(\N^l))$. Then $d_R=d_q$.

If $p>1$, then by the Banach-Alaoglu Theorem $(W,w_{p^*})$ is compact, which implies that $(W,w_{q^*})$ is compact since $w_{q^*}$ induces a weaker topology than $w_{p^*}$, as $\infty>q>p$ and hence $1<q^*<p^*$.
If $p=1$, we use the fact that $B(\ell^1(\N^{l}))$ is a closed and convex subset of $B(\ell^{q}(\N^{l}))$ and hence is weakly compact by the previous argument.
So by Theorem \ref{thm:compact banach v1} it suffices to show that $(W,d_q)/S_{\N}$ is totally bounded.
To this end choose $\eps>0$. 
Let $k=k(p,q,\eps/2)$ be the constant supplied by Lemma \ref{lem:approx}.
Denote by $[n]$ the set of positive integers $\{1,\ldots,n\}$ for any $n\in \N$.
Let $F$ be a finite set of points in $B(\ell^q([k^l]^l))$ such that for each $y\in B(\ell^q([k^l]^l))$ there exists $x\in F$ with $\|x-y\|_q\leq \eps/2$. 
By Lemma \ref{lem:approx}, for each $w\in W$ there exists a set $K\subset \N^l$ of size at most $k$ such that $\|w_K-w\|_q\leq \eps/2$. 
Then there exists $\sigma\in S_\N$ such that $\sigma(K)\subseteq [k^l]^l$, as the $l$-tuples in $K$ can contain at most $k^l$ distinct numbers.
This implies that $d_q(S_\N F,W)\leq \eps$ and finishes the proof.
\end{proof}

\begin{ack}
I thank the anonymous referees for pointing out some errors in a previous version of this paper as well as for comments on the topological part, leading to a significant improvement of the paper.
I also thank Lex Schrijver for comments and useful discussions.

The research leading to these results has received funding from the European Research Council
under the European Union's Seventh Framework Programme (FP7/2007-2013) / ERC grant agreement
n$\mbox{}^{\circ}$ 339109.
\end{ack}

\end{document}